\documentclass[10pt]{amsart}
\usepackage[margin=1.19in]{geometry}
\usepackage{amsmath,amsthm,amssymb,mathtools}
\usepackage{xspace,xcolor}
\usepackage[breaklinks,colorlinks,citecolor=teal,linkcolor=teal,urlcolor=teal,pagebackref,hyperindex]{hyperref}
\usepackage{tikz-cd}
\usepackage[all]{xy}

\usepackage{datetime}

\newtheorem{thm}{Theorem}[section]
\newtheorem{thmx}{Theorem}[section]

\newtheorem{prop}[thm]{Proposition}

\newtheorem{lem}[thm]{Lemma}
\newtheorem{cor}[thm]{Corollary}

\theoremstyle{definition}
\newtheorem{defn}[thm]{Definition}
\newtheorem{claim}{Claim}

\newtheorem*{condition*}{Condition}
\newtheorem*{conv*}{Convention}
\newtheorem{convention}[thm]{Convention}
\newtheorem{ques}{Question}
\newtheorem{notations}[thm]{Notations}

\theoremstyle{remark}
\newtheorem{ex}[thm]{Example}
\newtheorem{rmk}[thm]{Remark}
\newtheorem*{caution*}{Caution}

\newcommand{\C}{\mathbb C}
\newcommand{\Q}{\mathbb Q}
\newcommand{\Z}{\mathbb Z}
\newcommand{\Pp}{\mathbb P}

\newcommand{\Ll}{\mathcal L}

\newcommand{\sO}{\mathcal O}
\newcommand{\sstar}{($\star$)}

\newcommand{\FF}{\mathfrak F}
\newcommand{\comp}{\mathfrak c}
\newcommand{\pt}{\{\text{point}\}}

\newcommand{\bM}{\overline{\mathcal M}}
\newcommand{\mm}{\bM_{j,j',d,\chi}}
\newcommand{\cc}{\mathcal C_{j,j',d,\chi}}
\newcommand{\ff}{f_{j,j',d,\chi}}
\newcommand{\pp}{\pi_{j,j',d,\chi}}

\newcommand{\sss}{s_{j,j',d,\chi}}

\newcommand{\XX}{X_{j,j',d,\chi}}
\newcommand{\A}{A_{i,k;j,j',d,\chi}}

\newcommand{\unbroken}{unbroken map~}
\newcommand{\unbrokens}{unbroken maps~}

\newcommand{\Prin}{\mathop{\bf{Prin}}}

\newcommand{\NE}{\operatorname{NE}}
\newcommand{\vir}{\operatorname{vir}}
\newcommand{\PD}{\operatorname{PD}}
\newcommand{\x}{x}

\title{On Gromov-Witten theory of projective bundles}

\author[H.~Fan]{Honglu Fan}
\email{honglu.fan@math.ethz.ch}

\author[Y.-P.~Lee]{Yuan-Pin Lee}
\email{yplee@math.utah.edu}

\address{Department of Mathematics, University of Utah,
Salt Lake City, Utah 84112-0090, U.S.A.}

\begin{document}

\maketitle

\tableofcontents

\setcounter{section}{-1}

\section{Introduction}

\subsection{Statement of the main result}\label{s:0.1}
Let $Y$ be a smooth projective variety with a $T$-action and $V$ be a $T$-equivariant, \emph{not necessarily split}, vector bundle of rank $r$ over $Y$.
The projective bundle $\Pp(V)$ naturally carries an induced $T$-action. 
The equivariant cohomology of $\Pp(V)$ has the following presentation
\[
 H_T^*(\Pp(V)) = \frac{H^*_T(Y)[h]}{(c_T(V)(h))}
\]
where $h=c_1(\sO(1))$ and $c_T(V)(x) = \sum x^{r-i} c_i(V)$.
Given two such equivariant vector bundles $V_1, V_2$ over $Y$ with the same equivariant Chern classes  $c_T(V_1)=c_T(V_2)$, the equivariant cohomology/Chow are canonically isomorphic by the above presentation
\[
 \FF:H_T^*(\Pp(V_1))\cong H^*_T(\Pp(V_2)) .
\]
$\FF$ induces an isomorphism between the groups of numerical curve classes $N_1(\Pp(V_1))$ and $N_1(\Pp(V_2))$ by the intersection pairing,  and we slightly abuse the notation and denote the induced isomorphism also by $\FF$. 
This isomorphism on curve classes is uniquely characterized by the property: $(\FF D,\FF\beta)=(D,\beta)$ for any $D\in N^1(\Pp(V_1)), \beta\in N_1(\Pp(V_1))$.

\begin{thmx}[=Theorem \ref{main2}] \label{t:a}
If $Y$ is a projective smooth variety with a torus action such that there are finitely many fixed points and one dimensional orbits, 
then the $\FF$ induces an isomorphism of $T$-equivariant genus 0 Gromov-Witten invariants between $\Pp(V_1)$ and $\Pp(V_2)$
\[
 \langle\psi^{a_1}\sigma_1,\dotsc,\psi^{a_n}\sigma_n\rangle_{0,n,\beta}^{\Pp(V_1)}=\langle\psi^{a_1}\FF\sigma_1,\dotsc,\psi^{a_n}\FF\sigma_n\rangle_{0,n,\FF\beta}^{\Pp(V_2)} .
\]
\end{thmx}

Such $Y$ is often called a proper algebraic GKM manifold and examples include toric varieties, Grassmannians, flag varieties, and certain Hilbert schemes of points.

\subsection{Motivations}

It is natural to study Gromov--Witten theory of projective bundles, which has appeared in key steps in the relative Gromov--Witten theory (e.g., \cite{MP}) and in flops and $K$-equivalence (crepant transformation) conjecture (\cite{LLW1, LLW2, LLQW3}), among other things.
Furthermore, in the study of the \emph{functoriality} of the Gromov--Witten theory, the projective bundle also plays a role (\cite{LLW-string, LLW-ams}).
Namely, if one factors a projective morphism $f:X \to Y$ via embedding and smooth fibration
\[
\xymatrix{X \ar[r]^>>>>>\iota \ar[rd]_f & \mathbb{P}_Y(E) \ar[d]^\pi\\ & Y} 
\]
(where $E$ is a vector bundle over $Y$), 
one can study the general functoriality of Gromov--Witten theory by studying embedding (e.g., quantum Lefschetz hyperplane theorem) and projective bundles.

However, the study of non-split bundles proves to be difficult. 
For split bundles,  the fiberwise torus action and the virtual localization formula proves to be fairly effective general strategies to study the Gromov--Witten theory of the projective bundle, as in \cite{MP, GB, CGT}.
This does not work for the non-split bundles, however.
Instead, a degeneration argument is employed in \cite{LLQW3} and has had some success in a limited form of \emph{quantum splitting principle}.
Since the Chern roots of a nonsplit vector bundle can be represented as the first Chern classes of line bundles,
\emph{Theorem~\ref{t:a} can therefore be regarded as a first step towards a more complete form of quantum splitting principle}.
Once the bundle is split, the results in \cite{GB, CGT} give us a complete description of the Gromov--Witten theory of $\mathbb{P}(V)$ in all genera in terms of the genus zero theory of $Y$.

Note that if there exists an isomorphism  between $V_1$ and $V_2$ as complex vector bundles (but not necessarily as holomorphic ones), 
D.~Joyce informs us that $\mathbb{P}(V_1)$ and $\mathbb{P}(V_2)$ are symplectically deformation equivalent.
This is often not possible however. The simplest examples are given for $Y=\mathbb{P}^3$ in \cite{AR}.

\subsection{Discussions}

The result of this paper leads to the following questions.

\begin{ques}\label{question2}
Let $Y$ be a smooth variety and $V$ be a vector bundle. Are the Gromov-Witten invariants of $\Pp(V)$ uniquely determined by those of the base $Y$ and the total Chern class $c(V)$?
\end{ques}

If the above question seems hopelessly naive, the following might be more reasonable.

\begin{ques}\label{question3}
Is Question \ref{question2} true if one replaces the total Chern class $c(V)$ by the class in the algebraic K-group $[V]\in K^0(V)$?
\end{ques}

A note on the references: When the reference \cite{GB} is used, we specifically mean this ArXiv version of the paper, as we are not able to understand a key step in the published version.

\subsection{Acknowledgements}
The first author wish to thank Yuan Wang and Yang Zhou for helpful discussions, and to Andrei Musta\c{t}\u{a} and Sam Payne for answering related questions. 
An answer to the first author's posts in Mathoverflow is also proven useful. 
Both authors are partially supported by the NSF. The first author is also supported by SwissMAP and ERC-2012-AdG-320368-MCSK in the group of Rahul Pandharipande at ETH Z\"urich.

\section{Genus zero equivariant Gromov-Witten theory}

In this and the next section, we will recall some basic definitions and fundamental results in the equivariant Gromov--Witten theory of smooth projective varieties with torus action (possibly non-GKM). The study of a general torus action is necessary, because the projective bundle $\Pp(V)$ (c.f. Section \ref{s:0.1}) might not be a GKM manifold despite the base $Y$ being GKM.

\subsection{Equivariant twisted theory}
Let $T := (\C^*)^m$ be an algebraic torus and $X$ a smooth projective variety with a $T$ action. 
Let $E$ be a $T$-equivariant vector bundle on $X$. 
All the cohomology classes and characteristic classes are understood in the equivariant theories. 
We have $H^*_T(\text{point})\cong\C[C(T)]$ where $C(T)$ is the group of characters of $T$. 
Once a basis is chosen, $\C[C(T)]$ can be written as a polynomial ring $\C[\lambda_1,\dotsc,\lambda_m]$. 
The following notations are used throughout the paper:
\begin{itemize}
\item $R_T=H_T^*(\pt) = \C[\lambda_1,\dotsc,\lambda_m]$.
\item $S_T$ is the localization of $R_T$ by the set of non-zero homogeneous elements.
\end{itemize}

Let $\bM_{0,n}(X,d)$ be the moduli of stable maps with curve class $d\in \NE(X)$. 
The $T$-action on $X$ induces an $T$-action on $\bM_{0,n}(X,d)$ and we have an equivariant virtual fundamental class $[\bM_{0,n}(X,d)]^{\vir}$.
The universal family is naturally identified with the forgetful morphism $\pi_{0,n,d}:\bM_{0,n+1}(X,d)\rightarrow \bM_{0,n}(X,d)$ (forgetting the last marked point) and the universal stable map with $f_{0,n,d}:\bM_{0,n+1}(X,d)\rightarrow X$ (evaluating at the last marked point).
Let 
\[
 E_{0,n,d}=(\pi_{0,n,d})_* (f_{0,n,d}^*(E))\in K_0(\bM_{0,n}(X,d)).
\]
In the following, we \emph{assume} that $E_{0,n,d}$ is isomorphic in $K$-theory to a two-term complex of locally free sheaves
\[
 0\rightarrow E_{0,n,d}^0\rightarrow E_{0,n,d}^1\rightarrow 0,
\]
such that the equivariant Euler class $e_T(E_{0,n,d}^0)$ is invertible. Define
\[
e_T(E_{0,n,d})^{-1}=\displaystyle \frac{e_T(E_{0,n,d}^1)}{e_T(E_{0,n,d}^0)}, \quad (\text{and $e_T(E_{0,n,d}) =\displaystyle \frac{e_T(E_{0,n,d}^0)}{e_T(E_{0,n,d}^1)}$ when $e_T(E_{0,n,d}^1)$ is invertible}).
\]

\begin{defn}\label{twthy}
Let $\psi_i$ be the first Chern class of the universal cotangent line bundle. 
Assume $e_T(E_{0,n,d})^{-1}$ is defined. Given $\alpha_1, \dotsc, \alpha_n \in H^*_T(X)$, the {\it twisted equivariant Gromov-Witten invariant} is defined as
\[
\langle \psi^{k_1}\alpha_1, \dotsc, \psi^{k_n}\alpha_n \rangle_{0,n,d}^{X,E} = \int_{[\bM_{0,n}(X,d)]^{vir}} \displaystyle\frac{1}{e_T(E_{0,n,d})}\prod\limits_{i=1}^n \psi_i^{k_i}ev_i^*\alpha_i.
\]
\end{defn}

By convention the Euler class of zero sheaf is $1$ and the ``untwisted theory'' is therefore considered as a special case.

\begin{rmk}\label{invertible}
$e_T(E_{0,n,d})$ is often not invertible. To remedy this situation, an auxiliary $\C^*$ action, scaling the fiber of $E$ while acting trivially on $X$, is introduced. 
We have $R_{T\times\C^*}=R_T[\x]$ where $\x$ is the equivariant parameter corresponding to the $\C^*$ action. 
Let $R_T[\x,\x^{-1}]\!]$ be the ring of Laurent series in $\x^{-1}$. Let $(H^*_{T}(\bM_{0,n}(X,d))[\x])_\text{loc}$ be the localization of $H^*_{T}(\bM_{0,n}(X,d))[\x]$ by inverting monic polynomials in $\x$. 
It can be embedded in $H^*_{T}(\bM_{0,n}(X,d))\otimes_{R_T}R_T[\x,\x^{-1}]\!]$ by expanding the denominators. In this set-up,
\[
 e_{T\times\C^*}(E_{0,n,d})^{-1}= \displaystyle \frac{e_{T\times\C^*}(E_{0,n,d}^1)}{e_{T\times\C^*}(E_{0,n,d}^0)}
\]
is defined in $(H^*_{T}(\bM_{0,n}(X,d))[\x])_\text{loc}$. 
The corresponding twisted invariants are defined when the Euler class and the insertions are embedded into $H^*_{T}(\bM_{0,n}(X,d))\otimes_{R_T}R_T[\x,\x^{-1}]\!]$. The twisted invariants take values in $R_T[\x,\x^{-1}]\!]$. 

$R_T[\x,\x^{-1}]\!]$ is sometimes too big to be useful, and we need to make some changes.  
To start with, $\displaystyle \frac{e_{T\times\C^*}(E_{0,n,d}^1)}{e_{T\times\C^*}(E_{0,n,d}^0)}$ is an element in $(H^*_{T}(\bM_{0,n}(X,d))[\x])_\text{loc}$. 
In Section~\ref{projectivebundle}, we embed $\displaystyle \frac{e_{T\times\C^*}(E_{0,n,d}^1)}{e_{T\times\C^*}(E_{0,n,d}^0)}$ into $(H^*_{T}(\bM_{0,n}(X,d))[\x])_\text{loc}\otimes_{R_T}S_T$, and take $\x=0$ limit in $H^*_{T}(\bM_{0,n}(X,d))\otimes_{R_T} S_T$. The existence of the limit means that there are elements $\mathfrak e_0, \mathfrak e_1\in H^*_{T}(\bM_{0,n}(X,d))\otimes_{R_T} S_T$ with $\mathfrak e_0$ {\it invertible} such that
\[
 \mathfrak e_1 e_{T\times\C^*}(E_{0,n,d}^0) - \mathfrak e_0 e_{T\times\C^*}(E_{0,n,d}^1) \in \x H^*_{T}(\bM_{0,n}(X,d))[\x]\otimes_{R_T} S_T.
\]
The twisted invariants are defined by replacing ${e_T(E_{0,n,d})}^{-1}$ in Definition \ref{twthy} with $\displaystyle \frac{\mathfrak e_1}{\mathfrak e_0}$. Note that the invariants take values in $S_T$.
\end{rmk}

For any ring $R$, $R[\![Q]\!]$ shall always be understood as the ring of Novikov variables, i.e., the ring of power series of  $Q^d, d\in \NE(X)$ with coefficients in $R$. (See, e.g., \cite[Chapter~1]{LPbook}.)
The notations $R[\![z,z^{-1}]$ and $R[z,z^{-1}]\!]$ in this paper stand for the \emph{ring} of formal Laurent series, with inifinite powers in $z$ and $z^{-1}$ respectively (and finite in the other). 

\subsection{Lagrangian cone formulation}
This section serves the purpose of fixing the notations.
The readers are referred to \cite{giv1} for details of Givental's construction and various properties of the Lagrangian cones.

Let $\mathcal{H} = H^*_T(X)\otimes_{R_T}S_T[z,z^{-1}]\!][\![Q]\!]$. There is a bilinear pairing $(,)_E$ induced by the Poincar\'e pairing 
\[
 (\alpha,\beta)_E=\int_{[X]} e_T(E)^{-1}\alpha\cup\beta
\]
on $H^*_T(X)$. 
Let $\mathcal H_+=H^*_T(X)\otimes_{R_T}S_T[z][\![Q]\!], ~\mathcal H_-=z^{-1}H^*_T(X)\otimes_{R_T}S_T[\![z^{-1}]\!][\![Q]\!]$.
$\mathcal H$ admits a \emph{polarization} $\mathcal H=\mathcal H_+\oplus \mathcal H_-$ and is indeed naturally identified as the cotangent bundle $T^* \mathcal{H}_+$. 
The canonical symplectic structure of $\mathcal{H} \cong T^* \mathcal{H}_+$ is the $S_T[\![Q]\!]$-bilinear form 
\[
 \Omega(f(z),g(z))= \text{Res}_{z=0}(f(-z),g(z))_Edz
\] 
for any $f,g\in \mathcal H$. 

Givental's (twisted) Lagrangian cone is defined as the section of $T^* \mathcal{H}_+$ by the differential of genus-$0$ Gromov-Witten descendant potential, defined in a formal neighborhood at $-1z\in \mathcal H_+$.

More explicitly, let $\{\phi_i\}$ be a basis of $R_T$-module in $H^*_T(X)$ that induce a $\C$-basis in $H^*_T(X)/H^*_T(\pt)$. 
Assume $\phi_0=1$ and let $\phi^i$ be the dual basis in $H^*_T(X)\otimes_{R_T} S_T$ with respect to $(,)_E$. 
We parametrize cohomology classes by taking $t_k= \sum_i t_k^i\phi_i$ with $t_k^i$ formal variables. 
Let $t(z)=t_0+t_1z+\dotsb+t_k z^k+\dotsb$ be the formal parameters in $\mathcal H_+$, and $(\mathcal H,-1z)$ be the formal neighborhood at $-1z\in\mathcal H$.
 A general (formal) family of points $F \in(\mathcal H,-1z)$ in Givental's (twisted) Lagrangian cone $\Ll_{X,E}\subset(\mathcal H,-1z)$ is of the form
\begin{align*}
F(-z,t)&=-1z+t(z)+\sum\limits_{d\in \NE(X)_\Z}\sum\limits_{n=0}^\infty\sum\limits_{j=1}^{N} \displaystyle\frac{Q^d}{n!}\langle\frac{\phi_j}{-z-\psi},t(\psi),\dotsc,t(\psi)\rangle_{0,n+1,d}^{X,E}\phi^j \\
&=-1z+t(z)+\sum\limits_{d\in \NE(X)_\Z}\sum\limits_{n=0}^\infty\frac{Q^d}{n!}(\text{ev}_1)_*\left[\frac{e_T^{-1}(E_{0,n+1,d})}{-z-\psi_1}\prod\limits_{i=2}^{n+1}\text{ev}_i^*t(\psi_i) \right].
\end{align*}
The definition of Lagrangian cone depends on the polarization. 
Different polarizations will be employed as necessary. See Convention~\ref{polarization2}. 

\begin{rmk}
Note that the various functions  only exist in various suitably completed spaces of $\mathcal{H}$, rather than $\mathcal{H}$ itself.
The explanation of ``formal neighborhood" and a more rigorous definition of the theory can be made using formal schemes as in Appendix B of \cite{ccit} where it is painstakingly spelt out (and in different guises in \cite{LPbook}).
\end{rmk}

In addition, we define the $S$-matrix which will be used later.

\begin{defn}\label{def:S}
Let $\alpha, \tau\in H^*_T(X)\otimes_{R_T}S_T[\![Q]\!]$. Define the \emph{$S$-matrix} of the Gromov--Witten theory twisted by $E$ to be the following.
\[
S_\tau(-z)\alpha = \alpha+\sum\limits_{d\in \NE(X)_\Z}\sum\limits_{n=0}^\infty\sum\limits_{j=1}^{N} \displaystyle\frac{Q^d}{n!}\langle\frac{\alpha}{-z-\psi},\tau,\dotsc,\tau,\phi_j\rangle_{0,n+2,d}^{X,E}\phi^j.
\]
\end{defn}

\section{Virtual localization} \label{s:2}

Let $X$ be a smooth projective variety admitting an action by a torus $T=(\C^*)^m$. It induces an action of $T$ on $\bM_{0,n}(X,d)$. 
Let $\bM_\alpha$ be the connected components of the fixed loci in $\bM_{0,n}(X,d)^T$ labeled by $\alpha$ with the inclusion $i_\alpha: \bM_\alpha\rightarrow \bM_{0,n}(X,d)$. 
The virtual fundamental class $[\bM_{0,n}(X,d)]^{vir}$ can be written as \cite{GP}
\[
[\bM_{0,n}(X,d)]^{vir}= \sum_\alpha (i_\alpha)_* \displaystyle\frac{[\bM_\alpha]^{vir}}{e_T(N^{vir}_\alpha)} 
\]
where $[\bM_\alpha]^{vir}$ is constructed from the fixed part of the restriction of the perfect obstruction theory of $\bM_{0,n}(X,d)$, and the virtual normal bundle $N^{vir}_\alpha$ is the moving part of the two term complex in the perfect obstruction theory of $\bM_{0,n}(X,d)$.
Note that in our situation, the indices $\alpha$ are the decorated graphs defined in the Definition~\ref{decoratedgraph}.

\subsection{General facts on torus action}
In this subsection and the next, a few general facts about torus actions are stated, many without proofs.
These facts are straightforward consequences of results in \cite{bb} and \cite{MM}.
However, they are not expressed there in exactly the forms needed for our purpose and we have thus elected to collect them here.

Without further specification, by ``invariant'', we always mean $T$-invariant
and an {\it irreducible invariant curve} on $X$ is always assumed to be a proper reduced irreducible $1$-dimensional $T$-invariant subscheme of $X$.

$X^T\subset X$ denotes the fixed subscheme in $X$ under the torus action by $T$. Let $Z_1,\dotsb,Z_l$ be connected components of $X^T$.
Let $N_j:=N_{Z_j/X}$ be the normal bundle of $Z_j$ in $X$ and $\iota_j:Z_j\rightarrow X$ be the inclusion. 

Recall that 
\[
 C(T)=Hom_{\C^*}(T,\C^*)
\]
is the group of characters of $T$. We have the decomposition of the normal bundle $N_j$ on $Z_j$ into eigensheaves
\[
N_j=\bigoplus_{\chi\in C(T)}N_{j,\chi}.
\]
Let $\rho:\C^*\rightarrow T$ be a $1$-parameter subgroup. 

\begin{defn}
Define $(\rho,\chi)$ to be the weight of the composition $\C^*\xrightarrow{\rho} T\xrightarrow{\chi}\C^*$. 
\end{defn}

Since $N_j$ is of finite rank, for a generic $1$-parameter subgroup  $\rho$,  we have $(\rho,\chi)=0$ if and only if either $\chi=0$ or $N_{j,\chi}=0$.
%Applying \cite[Theorem~2]{bb} to $Z_j$, we have the following proposition.

\begin{prop}\label{prop:invariantsubsch}
There is a unique $T$-invariant reduced and irreducible closed subscheme $Z_j^\rho$ such that 
\begin{enumerate}
\item $Z_j\subset Z_j^\rho$, 
\item $Z_j^\rho$ is regular along $Z_j$, 
\item $N_{Z_j/Z_j^\rho}=\bigoplus\limits_{(\rho,\chi)>0} N_{j,\chi}$.
\end{enumerate}
\end{prop}

This can be seen by first restricting the $T$-action to the $1$-parameter subgroup determined by $\rho$. One can apply \cite[Theorem 4.1]{bb} and take $Z_j^{\rho}:=(X_j)^+$ (note that $(X_j)^+$ is under the notation of \cite{bb} when their $G$ is our $T$ and their $(X^G)_j$ is our $Z_j$). Notice that a point $x\in Z_j^{\rho}$ is characterized by $\lim_{\lambda\rightarrow 0} \rho(\lambda) \cdot x \in Z_j$. Since the group $T$ is commutative, $Z_j^{\rho}$ is also $T$-invariant.

Furthermore, by \cite[Theorem 2.5]{bb}, we are able to find a neighborhood $U$ of $Z_j$ in $Z_j^\rho$ that is locally isomorphic to $(Z_j\cap U)\times V$ with $V$ a $\C^*$ representation. We first of all use this decomposition to describe irreducible invariant curves in $X$. 
An irreducible invariant curve on $X$ must be either contained in a $Z_j$, or as the closure of a $1$-dimensional orbit. 
Any $1$-dimensional orbit $\mathfrak o$ is isomorphic to $\C^*$ as it is a quotient of a torus. 
Suppose $C$ intersect $Z_j$ at $p$. If we fix the isomorphism $\mathfrak o\cong \C^*$ in a way that the limit toward $0$ is $p$, then $T\rightarrow \mathfrak o\cong \C^*$ determines a character $\chi_p$. 
We call $\chi_p$ the \emph{(integral) character of an irreducible invariant curve near $p$} in this case.
If the irreducible invariant curve is in $X^T$, we set $\chi_p =0$.

Now given a $\rho$ as in Proposition~\ref{prop:invariantsubsch},and an irreducible invariant curve $C$ with $p\in C\cap X^T$.

\begin{prop}
If $(\rho,\chi_p)>0$, $C$ is contained in $Z_j^\rho$.
\end{prop}
\begin{proof}
$X$ can be covered by $T$-invariant affine open sets by \cite{sh}. 
Let $U$ be such an open neighborhood of $p$ and let $I_C\subset \C [U]$ be the ideal of $C \cap U \subset U$. 
Since $U$ is $T$-invariant, $\C [U]$ is graded by the characters lattice $C(T)$, and $I_C$ is a homogeneous ideal. 
Because $C \setminus \{p\}$ is isomorphic to $\C^*$ on which $T$ acts by $\chi_p$, $\C [U]/I_C\subset \C [C \setminus \{p\}]\cong \C [t,t^{-1}]$ where $t$ is a homogeneous element graded by $\chi_p$. 
Because the limit $t\rightarrow 0$ exists (which is $p$), $\C [U]/I_C$ has nonempty $a . \chi_p$-graded piece only if $a>0$. 
In other words, the image of $\C [U]/I_C\subset \C [t,t^{-1}]$ in fact lies in $\C [t]$.

By the construction in \cite{bb}, $Z_j^{\rho}$ is cut out by homogeneous elements $u\in \C [U]$ graded by $\chi'$ such that $(\rho,\chi')<0$. 
Now the image of $u$ in $\C [U]/I_C$ must be zero, as $(\rho,\chi_p)>0$ and the fact that $\C [U]/I_C$ has nonempty $a\chi_p$-graded piece only if $a>0$. 
Therefore any such $t$ lies in $I_C$, i.e., $C$ is a closed subscheme of $Z_j^{\rho}$. 
\end{proof}

\begin{cor} \label{c:2.5}
Let $C$ be the closure of a $1$-dimensional orbit $\mathfrak o\cong\C^*$. The limits of $0,\infty$ land in different connected components $Z_j$, $Z_{j'}$ of $X$. 
Furthermore, in a neighborhood of $ p \in Z_j  \cap C$, the irreducible invariant curve $C$ can be parameterized by $t$ as $(c_1 t^{a_1}, \ldots, c_nt^{a_n})$. 
\end{cor}

\begin{proof}
The first part is obvious. There is an open neighborhood $U$ of $p$, such that $Z_j^\rho \bigcap U$ is some open set times a $\C^*$ representation $V$. 
Since $C$ is invariant, $C\bigcap U\subset \{p\}\times V$. 
As in the previous proof, embed $\C [V]/I_C\hookrightarrow k[t]$. 
Let $x_1,\dotsc,x_n$ be linear functions in $\C [V]$ that are homogeneous with respect to $\C^*$ action. 
The image of $x_i$ in $\C [t]$ must be homogeneous. Thus it is $c_i t^{a_i}$ for some $c_i\in\C, a_i\in\Z_{>0}$. 
\end{proof}
In particular, we have

\begin{cor}
An irreducible invariant curve is homeomorphic to its normalization.
\end{cor}

%Fix $Z_j$ and choose a $\Q$-basis in the cocharacter lattice $\{ \rho_1,\dotsc,\rho_m \}$ such that
%$(\rho_i,\chi)=0$ for any $1\leq i\leq m$ iff $\chi=0$ or $N_{j,\chi}=0$.
%Write $\underline{\rho}=(\rho_1,\dotsc,\rho_m)$. 
%We construct $Z_j^{\rho_1} \subset X$ as above, and repeat the process for $\rho_2$ in $Z_j^{\rho_1}$ to get $Z_j^{(\rho_1,\rho_2)} \subset Z_j^{\rho_1}$, and so on.
%Eventually we get a subscheme $Z_j^{\underline{\rho}}$ regular along $Z_j$. Invoking \cite[Theorem 2.5]{bb} again, we have

%\begin{prop}\label{localstructure}
%Let $\underline{\rho}=\{\rho_1,\dotsc,\rho_m\}$ be given as above. There is an invariant subscheme $Z_j^{\underline{\rho}}$ regular along $Z_j$.
%Furthermore there is an open neighborhood $U$ of $Z_j$ in $Z_j^{\underline{\rho}}$ such that it is $T$ equivariantly isomorphic to $\bigoplus\limits_{{\chi}, \, {(\rho_i,\chi)>0, \, \forall i}} N_{j,\chi}$. 
%Additionally, given a character $\chi_0$ such that $(\rho_i,\chi_0)>0, 1\leq \forall i\leq m$, any irreducible invariant curve having character $\chi_0$ near $Z_j$ will be contained in $Z_j^{\underline{\rho}}$
%\end{prop}

\subsection{$T$-invariant stable maps}
We now proceed to study genus zero invariant stable maps. 
The contents of this subsection follows from results in \cite{MM} (and \cite{IB}).
The $g=0$ assumption is tacitly imposed throughout the paper, although most of the results hold in general.

Define $C(T)_\Q:=C(T)\otimes_\Z \Q$.
Let $f:(C,x_1,\dotsc,x_n)\rightarrow X$ be an \emph{invariant} stable map and $C_0$ an irreducible component of $C$.
Let $p\in C_0$ be a point such that $f(p)\in X^T$.

\begin{defn}
 Let $\chi_{f(p)}$ be the integral character of the irreducible invariant curve $f(C_0)$ at $f(p)$ and $k$ be the degree of $f|_{C_0}:C_0\rightarrow f(C_0)$. 
Define the \textit{fractional character of $C_0$ at $p$} to be $\chi_p=\chi_{f(p)}/k\in C(T)_\Q$. In particular when $f(C_0)\subset X^T$, $\chi$ is $0$.
\end{defn}

\begin{rmk}
We would like to emphasize again that there is a big difference of definitions between the \emph{integral character} at a fixed point on a one-dimensional orbit closure, and the \emph{fractional character} at a point on the domain curve of an invariant stable map. 
\end{rmk}

\begin{prop}[\cite{MM}] \label{invariance}
The fractional character $\chi_p$ defined above is deformation invariant for $T$-\emph{invariant} stable maps.
More explicitly, let $f:(\mathcal C,x_0,\dotsc,x_n)\rightarrow X$ be a family of invariant stable maps over a connected base $S$ and $x_i:S\rightarrow \mathcal C$ the sections to the corresponding markings. Let $(C_p)_0$ be the unique irreducible component of the fiber $C_p$ containing $x_0(p)$. 
Then the fractional character of $(C_p)_0$ at $x_0(p)$ does not depend on the choice of $p\in S$.
\end{prop}

\begin{rmk}
In this proposition, stability is crucial. One can construct a counterexample for semistable curves by making a family of curves specialized into a curve with $x_0$ lying on a contracted component.
\end{rmk}

\begin{prop}[\cite{MM}] \label{deformation}
Let $f:(C,x_1,\dotsc,x_n)\rightarrow X$ be an invariant stable map. Let $C_1, C_2$ be two irreducible components of $C$ such that $p \in C_1 \cap C_2$ is a node. 
The fractional characters of $C_1, C_2$ at $p$ are denoted as $\chi_1, \chi_2$ respectively. 
Suppose $C_1,C_2$ are not contracted under $f$, then the node $p$ can be smoothened in a family of invariant stable maps only if $\chi_1+\chi_2=0$.
\end{prop}

The condition $\chi_1+\chi_2=0$ is consistent with \cite[Definition 2.7(3)]{MM} and is a necessary condition for smoothability within a component of the fixed substack $\bM_{0,n}(X,d)^T$.
In view of this, we make the following definition.

\begin{defn}
Notations are the same as the above proposition. $x$ is said to be a node satisfying condition \sstar~if
\begin{itemize}
  \item[{\sstar}] \qquad $\chi_1\neq 0 \, \text{and} \, \chi_1+\chi_2=0.$
\end{itemize}
\end{defn}

We would like to introduce an important type of invariant stable maps whose moduli spaces are main building blocks in the components of $\bM_{0,n}(X,d)^T$, cf.\ \cite[\S~3]{MM} and \cite[\S~7.4]{IB}.

\begin{defn}
Given two fixed loci $Z_j,Z_{j'}$, by an {\it \unbroken of a nonzero fractional character $\chi\in C(T)_\Q$ between $Z_j, Z_{j'}$}, we mean an invariant stable map with $2$-markings $f:(C,x_+,x_-)\rightarrow X$ such that
\begin{enumerate}
\item $C$ is a chain of $\Pp^1$, with irreducible components $C_1,\dotsc,C_l$ such that $C_i\cap C_{i+1}\neq\emptyset$;
\item $x_+\in C_1$, $x_-\in C_l$. Also $f(x_+)\in Z_j$, and $f(x_-)\in Z_{j'}$;
\item the fractional character of $C$ at $x_+$ is $\chi$, and all nodes in $C$ satisfy \sstar.
\end{enumerate}
\end{defn}

Moduli spaces of \unbrokens are used to describe the fixed substack of $\bM_{0,n}(X,d)$ under the torus action. We list all related definitions and notations below for later references.

\medskip
\begin{notations}~ \label{n:2.14}
\begin{itemize}
\item Given $Z_j$, $Z_{j'}$ and a curve class $d\in \NE(X)$, $\mm$ denotes the closed substack of $\bM_{0,2}(X,d)$ parametrizing \unbrokens of fractional character $\chi$ between $Z_j$ and $Z_{j'}$ whose curve class is $d$.
\item Denote $\pp:\cc\rightarrow\mm$ the universal family and $\ff:\cc\rightarrow X$ the corresponding universal stable map. 
\item Let $[\mm]^{vir}$ be the virtual class induced by the restriction of the natural perfect obstruction theory from $\bM_{0,2}(X,d)$.
\item There are two sections $\sss^+, \sss^-:\mm\rightarrow\cc$ mapping an \unbroken $f:(C,x_+,x_-)\rightarrow X$ between $Z_j,Z_{j'}$ to the markings $x_+$ and $x_-$ respectively. 
\item $\psi_+=(\sss^+)^*\omega_{\pp}$, $\psi_-=(\sss^-)^*\omega_{\pp}$, where $\omega_{\pp}$ is the relative dualizing sheaf.
\item $\text{ev}_+=\ff \circ \sss^+:\mm\rightarrow Z_j$ and $\text{ev}_- =\ff \circ \sss^-$.
\item Let $\XX \subset X$ be the image of $\ff(\cc)$.
\end{itemize}
\end{notations}

\begin{ex}
Let $T=\C^*$ acts on $\Pp^2$ by sending $[x_0;x_1;x_2]$ to $[\lambda^{-2} x_0;\lambda^{-1}x_1;x_2]$. By a slight abuse of notation, also let $\lambda$ be the identity character of $\C^*$. Let $Z_0=[1;0;0], Z_1=[0;1;0], Z_2=[0;0;1]$. The moduli of \unbrokens between $Z_0$ and $Z_2$ with character $\lambda$ forms a $1$-dimensional family that sweeps through $\Pp^2$. A generic member of such \unbrokens has irreducible domain curve, but a special one ($[a;b;0], a,b\in\C$) breaks in $Z_1$ with condition \sstar~ satisfied.

Let $l$ be the class of a line in $\Pp^2$. In this example, $\bM_{0,1,l,\lambda}=\{\text{pt}\}, \bM_{0,2,l,\lambda}=\Pp(1,2)$. For $\bM_{0,2,l,\lambda/k}$, they parametrize degree $k$ covers of \unbrokens between $Z_0, Z_2$. Although the coarse moduli spaces are the same, we have $\bM_{0,2,l,\lambda/k}=\Pp(k,2k)$.
\end{ex}

\subsection{Decorated graphs and moduli of invariant stable maps}\label{graph}
 
We associate graphs to combinatorial types of fixed loci $\bM_{0,n}(X,d)^T \subset \bM_{0,n}(X,d)$, following the notations in \cite[Definition~52]{liu}. 

\begin{defn}\label{decoratedgraph} A decorated graph $\vec{\Gamma}=(\Gamma,\vec{p},\vec{d},\vec{s},\vec{\chi})$ for a genus-$0$, $n$-pointed, degree $d\in \NE(X)$ invariant stable map consists of the following data.

\begin{itemize} 
\item $\Gamma$ a finite connected graph, $V(\Gamma)$ the set of vertices and $E(\Gamma)$ the set of edges;
\item $F(\Gamma)=\{(e,v)\in E(\Gamma)\times V(\Gamma)~|~v~\text{incident to}~ e\}$ the set of flags;
\item the label map $\vec{p}:V(\Gamma)\rightarrow \{1,\dotsc,l\}$;
\item the degree map $\vec{d}:E(\Gamma)\cup V(\Gamma)\rightarrow \NE(X)\cup\mathop{\bigcup}_j\NE(Z_j)$, where $\vec{d}$ sends a vertex $v$ to $\NE(Z_{\vec{p}(v)})$ and an edge $e$ to $\NE(X)$;
\item the marking map $\vec{s}:\{1,2,\dotsc,n\}\rightarrow V(\Gamma)$ for $n>0$;
\item the fractional character map $\vec{\chi}:F(\Gamma)\rightarrow C(T)_\Q$.
\end{itemize}
They are required to satisfy the following conditions:
\begin{itemize}
\item $V(\Gamma), E(\Gamma), F(\Gamma)$ determine a connected graph. 
\item $\sum\limits_{e\in E(\Gamma)}\vec{d}(e)+\sum\limits_{v\in V(\Gamma)}\vec{d}(v)=d$.
\item Let $v$ be a vertex such that $\vec{d}(v)=0$, $\vec{s}^{-1}(v)=\emptyset$ and there are exactly two edges incident to $v$. The two edges $e_1, e_2$ incident to $v$ must satisfy $\vec{\chi}((e_1,v))+\vec{\chi}((e_2,v))\neq 0$. 
\end{itemize}
\end{defn}

We associate a decorated graph $\vec{\Gamma}$ to an invariant stable map $f$ in the following way.

\noindent \textbf{Vertices}:
\begin{itemize}
\item The connected components in $f^{-1}(X^T)$ are either curves or points. Assign a vertex $v$ to a connected component $\comp_v$ in $f^{-1}(X^T)$ that is either a sub-curve of $C$, a smooth point in $C$, or a node in $C$ that does \emph{not} satisfy condition \sstar.
\item Define $\vec{p}(v)=j$ where $1\leq j\leq l$ is the label such that $f(\comp_v)\in Z_j$.
\item Define $\vec{d}(v)=[\comp_v]\in \NE(Z_j)$ the numerical class $f_*([\comp_v])$ (notice if $\comp_v$ is a point, $[\comp_v]=0$).
\item For $i=1,\dotsc,n$, define $\vec{s}(i)=v$ if $x_i\in \comp_v$.
\end{itemize}
\noindent \textbf{Edges}:
\begin{itemize}
\item Assign each component of $C \setminus \mathop{\bigcup}\limits_{v\in V(\Gamma)}\comp_v$ an edge $e$. Let $\comp_e$ be the closure of the corresponding component.
\item We write $\vec{d}(e)=f_*[\comp_e]\in \NE(X)$.
\end{itemize}
\noindent \textbf{Flags}:
\begin{itemize}
\item $F(\Gamma)=\{(e,v)\in E(\Gamma)\times V(\Gamma)~|~\comp_e\cap \comp_v\neq\emptyset\}$
\item Given a flag $(e,v)$, $\comp_e\cap \comp_v$ must consist of a single point $x$. Let the fractional character of $\comp_e$ at $x$ be $\chi$. Define $\vec{\chi}((e,v))=\chi$.
\end{itemize}

\begin{rmk}\label{reducermk}
In the case of toric variety, for any invariant stable map, none of the node in $f^{-1}(X^T)$ can satisfy condition \sstar. Thus we reduce back to the traditional definition of decorated graphs in, for example, \cite{liu}.
\end{rmk}

\begin{conv*}
To further shorten the expressions, we adopt the following convention when the decorated graph $\vec\Gamma$ is given in the context.
\begin{itemize}
\item $p_v=\vec p(v)$,
\item $s_i=\vec s(i)$,
\item $d_v=\vec d(v)$, $d_\Gamma=\sum\limits_{e\in E(\Gamma)}\vec{d}(e)+\sum\limits_{v\in V(\Gamma)}\vec{d}(v)$,
\item $\chi_{e,v}=\vec{\chi}((e,v))$.
\end{itemize}
\end{conv*}

Since edge curves with nontrivial $T$-action can only degenerate to create nodes satisfying condition \sstar, the associated graph is invariant under deformations.

\begin{defn}
Given a vertex $v\in V(\Gamma)$, define $E_v=\{e\in E(\Gamma)|(e,v)\in F(\Gamma)\}$ to be the set of edges containing $v$. Also define $S_v=\vec{s}^{-1}(v)\subset \{1,\dotsc,n\}$. 
Define $\text{val}(v)=|E_v|$ to be the valence (without counting the number of markings) of $v$, and $n_v=|S_v|$ to be the number of marked points on $v$.
\end{defn}

An edge $e$ must connect two vertices. We call the two vertices $v_1(e),v_2(e)$ whose order is arbitrary. Given a vertex $v$, we write $\bM_{0,E_v\cup S_v}(Z_j,d)$ instead of $\bM_{0,\text{val}(v)+n_v}(Z_j,d)$ if we want to index the marked points by edges in $E_v$ and markings in $S_v$. Under such notations, the corresponding evaluation map at marked point labeled by $e\in E_v$ is denoted by $\text{ev}_e$.

Given a decorated graph $\vec{\Gamma}$, one can construct the corresponding fixed component in $\bM_{0,n}(X,d)$ out of fiber products of $\bM_{j,j',d,\chi}$ and $\bM_{0,n_i}(Z_j,d)$. Let $\vec{\Gamma}$ be a decorated graph. Define 
\[
\bM_{\vec{\Gamma}}=\left[ \prod\limits_{v\in V(\Gamma)} \bM_{0,E_v\cup S_v}(Z_{p_v},d_v)\right ]\times'\left [\prod\limits_{e\in E(\Gamma)} \bM_{p_{v_1(e)},p_{v_2(e)},d_e,\chi_{e,v_1(e)}}\right ]
\]
where we make the convention that $\bM_{0,n}(Z_j,0):=Z_j$ for $n=1,2$. Here the $\times'$ is the fiber product defined as follows. For any flag $(e,v)$, let $v'$ be the other vertex on $e$ different from $v$. The $\times'$ identifies the evaluation map $\text{ev}_e$ in $\bM_{0,E_v\cup S_v}(Z_{p_v},d_v)$ with $ev_+$ in $\bM_{p_v,p_{v'},d_e,\chi_{e,v_1(e)}}$. 

Let $\text{Aut}(\Gamma)$ be the automorphism of graphs. Note that automorphism coming from multiple cover (denoted by $k$) is already accounted in $\mm$. We have:
\[
[\bM_{\vec{\Gamma}}/\text{Aut}(\Gamma)]\subset \bM_{0,n}(X,d)^T.
\]
Moreover, $[\bM_{\vec{\Gamma}}/\text{Aut}(\Gamma)]$ is a union of connected components of $\bM_{0,n}(X,d)^T$.

\section{A recursion relation}

\subsection{Statement}

In this subsection, we state a technical result (Theorem \ref{main}), generalizing A.~Givental's theorem in \cite[Theorem~2]{GB}.
It is a way of packaging the virtual localization into a recursive form and is applied in Section~\ref{projectivebundle} to prove our main result on Gromov--Witten theory of projective (non-split) bundles. 
Before the statement of the theorem, a few definitions are in order.

Recall that, given a meromorphic function $f(z)$, the \emph{principal part} at $z=a$, denoted $\Prin\limits_{z=a}f$, is a polynomial in $1/(z-a)$ without constant term such that $f-\Prin\limits_{z=a} f$ has no pole at $z=a$.

Let $F$ be a family of points in the Lagrangian cone introduced earlier and let $F^j:=\iota_j^*F$. Localization formula implies that
\begin{equation}\label{loc}
F=\sum\limits_{j=1}^l (\iota_j)_!\displaystyle\frac{F^j}{e_T(N_j)} .
\end{equation}
When we restrict via $\iota_j$, the following function spaces and polarizations are used.

\begin{convention}\label{polarization2}
Let $\mathcal H^j$ be
\[
 \mathcal{H}^j :=  H^*(Z_j;S_T) \left[ z, \frac{1}{z} \right] \left[ \mathfrak S \right] [\![Q]\!] ,
\]
where $\displaystyle \mathfrak S=\left\{ \frac{1}{z+\chi} \right\}_{\chi\in C(T)_{\Q}}$.
The polarization is given by $\mathcal{H}^j =  \mathcal{H}^j_+ \oplus  \mathcal{H}^j_-$ such that $\mathcal{H}^j_+$ has no $z^{-1}$ (but has $(z+\chi)^{-1}$) and 
\[\mathcal H^j_-=z^{-1}H^*(Z_j;S_T)[z^{-1}][\![Q]\!].\]
The Lagrangian cone in the function space $\mathcal H^j$ under the above polarization is denoted by $\Ll^j_E$. 
\end{convention}

\begin{rmk}
Even though the above definition of $\mathcal{H}^j$ appears to involve polynomials of infinite indeterminates, in practice this is not needed.
Given a torus action on a fixed $X$, there are only finitely many integral characters $\chi \in C(T)$ associated to all one-dimensional orbits in $X$. 
For a fixed $d$ in $Q^d$, the \emph{fractional characters} involved is of the form $\chi/k$, where $\chi$ are the integral characters of one-dimensional orbits, and $k$ is linearly bounded by the numerical class $d$.
\end{rmk}

A priori we do not have $F^j\in (\mathcal H^j,-1z)$, but the localization computation in the next subsection allows us to write $F^j$ as an element in the formal neighborhood $(\mathcal H^j,-1z)$.
$F$ as a Laurent series in $z^{-1}$ can be recovered using localization formula \eqref{loc} by expanding $(z+\chi)^{-1}$ as power series in $z^{-1}$. 
We note that $F^j$ satisfies the following properties.
\begin{itemize}
\item The Novikov variables $Q^d$ still have exponents $d\in \NE(X)$.
\item The coefficient of $Q^d$ for a fixed $d$ is a polynomial instead of a power series in $z^{-1}$.
\end{itemize}

\begin{rmk}\label{decomp}
The polarization $\mathcal{H}^j =  \mathcal{H}^j_+ \oplus  \mathcal{H}^j_-$ is a form of partial fraction decomposition.
That is, any element $p(z)\in \mathcal H^j$ can be uniquely decomposed into
\[
 p(z)=[p(z)]_++[p(z)]_\bullet+[p(z)]_-,
\]
where 
\begin{align*}
[p(z)]_+ &\in H^*(Z_j;S_T)[z][\![Q]\!],\\
[p(z)]_\bullet &\in \bigoplus\limits_{\chi\in C(T)_\Q}(z+\chi)^{-1}H^*(Z_j;S_T)[(z+\chi)^{-1}][\![Q]\!], ~\text{and} \\
[p(z)]_- &\in z^{-1}H^*(Z_j;S_T)[z^{-1}][\![Q]\!].
\end{align*}
This also induces a decomposition of elements in $(\mathcal H^j,-1z)$.
\end{rmk}

\begin{rmk}
We use notation $\Ll^j_E$ other than $\Ll_{Z_j,E}$ to emphasize the different underlying function spaces and polarizations. 
Under this new polarization, the points in $\Ll^j_E$ allow $(z+\chi)^{-1}$ terms in $t(z)$, i.e.,
\[
 t(z)\in H^*(Z_j;S_T)[\![t_i^1,\dotsb,t_i^N]\!][z,(z+\chi)^{-1}][\![Q]\!], ~i\geq 0, ~\chi\in C(T)_\Q, ~N=\text{dim}(H^*(Z_j;\C)) .
\]
\end{rmk}

We now state the recursion relation. 

\begin{thm}\label{main}
Assume all $e_T(E_{0,n,d})^{-1}$ involved are defined. There exist equivariant cohomology classes $\A\in H_T^*(\mm)\otimes_{R_T}S_T$ such that (a formal family of point) $F\in(\mathcal H,-1z)$ lies in $\Ll_{X,E}$ if and only if its restrictions $F^j$ can be written as elements in $(\mathcal H^j,-1z)$ and satisfy the following conditions:
\begin{enumerate}
\item[(1)] $F^j(-z)\in\Ll^j_{N_j\oplus E}$ (under the indicated polarization);
\item[(2)] The principal parts of $\{F^j(-z)\}$ satisfy recursively defined expressions:
\[
\Prin\limits_{z=-\chi}F^j(z)=\sum\limits_{i, k, j', d, \chi}(ev_+)^{vir}_*\left[ \frac{Q^{d}\A\left.\displaystyle\frac{\partial^k}{\partial w^k}\left[(ev_-)^* \left( F^{j'}(w)-\Prin\limits_{w=-\chi} F^{j'}(w) \right) \right]\right |_{w=-\chi}}{(z+\chi)^i} \right] 
\]
where $ev_-$ is the map defined on $\mm$.
\end{enumerate}
Furthermore, $\A\in H_T^*(\mm)$ are uniquely determined by
\begin{itemize}
\item[(i)] the classes of equivariant vector bundles $[(TX\oplus E)|_{\XX}]\in K^0_T(\XX)$, 
\item[(ii)] the classes of equivariant vector bundles $[N_j\oplus E]\in K^0_T(Z_j)$,
\item[(iii)] the moduli stacks $\mm$, their virtual classes, and their universal families $\cc$.
\end{itemize}
\end{thm}

Condition (2) is sometimes referred to as the {\it ``recursion relation"}. The recursion relation comes from virtual localization, cf., Equation \eqref{tjx1} in the proof below.

\begin{rmk} \label{r:finitechi}
The above sum for $\Prin\limits_{z=-\chi}F^j(z)$ is finite in $i, k, j', \chi$ for a fixed $d$ when the target variety is of finite type. 
That is, for a fixed $d$, all but finitely many $\A$ vanish.
This follows from the proof given below.
\end{rmk}

We state a corollary to be used in the last section. Let $X$ be a smooth variety with a torus actions by $T$. Let $W$ be a smooth invariant subvariety of $X$.
Assume $e_T((N_{W/X})_{0,n,d})^{-1}$ are defined (or defined as a limit according to Remark \ref{invertible}). 
We compare the recursion relations.

\begin{cor}\label{A}
$X,W$ as above. Fix a $\chi\in C(T)_\Q$. Fix a $Z_j\subset W$. 
Suppose for any $1\leq j'\leq l$, $d\in \NE(X)$, we have either $\mm=\emptyset$ or $\XX\subset W$. 
Then $\A$ in the Theorem \ref{main} for untwisted theory on $X$ equal to those of the twisted theory by $N_{W/X}$ on $W$.
\end{cor}

\begin{proof}

Observe that $TW\oplus N_{W/X}$ and $N_{Z_j/W}\oplus N_{W/X}|_{Z_j}$ can be deformed to $TX|_{W}$ and $N_{Z_j/X}$ respectively. 
Hence they represent the same elements in the $K$-groups. All inputs involved in (i-iii) for $\A$ are therefore identified.
\end{proof}

The proof of Theorem \ref{main} occupies the next subsection.

\subsection{Proof of Theorem \ref{main}}\label{localization}
The following proof is parallel to the proof of Givental's theorem in \cite[Theorem 2]{GB}. 

For notational convenience, introduce
\[
  (ev_+)^{\vir}_* (\alpha) := \PD\left[ (ev_+)_*(\alpha\cap[\mm]^{vir})\right]
\]
where $\PD$ is the Poincar\'e duality map (The case for $ev_-$ on $\mm$ or $ev_i$ on $\bM_{0,n}(X,d)$ should be understood accordingly).

Recall $F(-z)\in \Ll_E$, and $F^j(-z):=\iota_j^*F(-z)$ is, by definition,
\begin{equation}\label{Fj}
F^j(-z,t)=-1z+\iota_j^*t+\iota_j^*\sum\limits_{n,d}\displaystyle\frac{Q^d}{n!}(\text{ev}_1)^{\vir}_*\left[\frac{e_T^{-1}(E_{0,n+1,d})}{-z-\psi_1}\prod\limits_{i=2}^{n+1}\text{ev}_i^*t(\psi_i)\right] .
\end{equation}
Later the variable $t$ in ``$F^j(-z,t)$" is suppressed and change of variables (for example $F^j(w)$) occur only in the ``$-z$" variable. We evaluate the above expression by virtual localization. 

Let $\vec{\Gamma}$ be a decorated graph. 
and let $N_{\vec{\Gamma}}^{\text{vir}}$ be the virtual normal bundle on $\bM_{\vec{\Gamma}}$. Then
\[
F^j=-1z+\iota^*_j t+\sum\limits_{\begin{subarray}{c} \vec{\Gamma} ~~\text{with} \\p_{s_1}=j \end{subarray}}\displaystyle\frac{1}{\text{Aut}(\Gamma)}\frac{Q^d}{n!}\iota_j^*(\text{ev}_1)^{\vir}_* \left[\frac{e_T^{-1}(E_{0,n+1,d}\oplus N_{\vec{\Gamma}}^{vir})}{-z-\psi_1}\prod\limits_{i=2}^{n+1}\text{ev}_i^*t(\psi_i)\right] ,
\]
where $s_1$ is the vertex where the first marking lies (see Definition \ref{decoratedgraph}). Given a decorated graph $\vec\Gamma$ with $\sum\limits_{v\in V(\Gamma)}d_v=d$, write the contribution of $\vec{\Gamma}$ as
\[
\text{Cont}_{\vec\Gamma}(z)= \displaystyle\frac{1}{n!\text{Aut}(\Gamma)}\iota_j^*(\text{ev}_1)^{\vir}_* 
  \left[\frac{e_T^{-1}(E_{0,n+1,d}\oplus N_{\vec{\Gamma}}^{vir})}{-z-\psi_1}\prod\limits_{i=2}^{n+1}\text{ev}_i^*t(\psi_i)\right] .
\]
Since $\bM_{\vec\Gamma}$ is a fiber product of a collection of $\mm$ and $\bM_{0,E_v\cup S_v}(Z_{p_v},d_v)$, this graph contribution can be computed by integrals on each of these spaces with suitable pull-backs and push-forwards. We will see in a moment that there is a recursive structure in the summation over graphs $\vec{\Gamma}$ such that $p_{s_1}=j$ (i.e., the first marked point lies on $Z_j$).  

Consider first the contribution from graphs when the vertex $s_1$ is incident to a single edge with fractional character $\chi\in C(T)_\Q$. 
In this case, ${s_1}$ must be connected to another vertex $v'$ such that $p_{v'}=j'$ for some $j'\neq j =p_{s_1}$. 
Let $e$ be the edge connecting $s_1$ and $v'$, $d_{s_1} := d_e \in \NE(X)$ be the degree associated to the edge $e$, and $\chi=\chi_{e,s_1}$ be the fractional character at $s_1$.

Introduce $t^{j,\chi}(z)$ as
\begin{align*}
&t^{j,\chi}(z) := \sum\limits_{\begin{subarray}{c} \vec{\Gamma} ~~\text{with} \\ d_{s_1}=0, \, \chi_{e,s_1}=\chi \\ p_{s_1}=j,\,  n_{s_1}=0, \, \text{val}(s_1)=1\end{subarray}}
   Q^{d_\Gamma}\text{Cont}_{\vec\Gamma}(z)\\
& = \sum\limits_{j'\neq j,d} \iota_j^*(\text{ev}_+)^{vir}_* \left[ \displaystyle\frac{Q^d\text{Cont}_{(e,s_1)}}{-z-\psi_++\chi} \sum\limits_{\begin{subarray}{c} \vec{\Gamma} ~~\text{with}~d_{s_1}=0,\\ \chi_{e,s_1}\neq\chi, p_{s_1}=j' \end{subarray}}  Q^{d_\Gamma}\left\{ (\text{ev}_-)^*\text{Cont}_{\vec\Gamma}(w) \right\}_{w=\psi_--\chi} \right] .
\end{align*}
Expanding the sum over graphs, we get
\begin{equation}\label{tjx}
t^{j,\chi}(z) = \sum\limits_{j'\neq j,d} \iota_j^*(\text{ev}_+)^{vir}_* \left[ \displaystyle\frac{Q^d\text{Cont}_{(e,s_1)}}{-z-\psi_++\chi}  \left\{ (\text{ev}_-)^*\left[F^{j'}(w)-t^{j',\chi}(w) \right] \right\}_{w=\psi_--\chi} \right]
\end{equation}
where the $ev_+, ev_-$ are the maps from $\mm$, and the contribution
\[
 \text{Cont}_{(e,s_1)}=\displaystyle \displaystyle\frac{ e_T(\left[\pp)_!T_{\pp}(-D_+)\right]^m) }{e_T(\left[(\pp)_!\ff^*(TX\oplus E)\right]^m)} .
\]
In the above, $T_{\pp}$ is the relative tangent bundle. $D_+$ is the divisor corresponding to the image of $\sss^+$. Superscript $m$ means the moving part, i.e., the subsheaf generated by homogeneous elements of nontrivial $T$-character (\cite{GP}). In writing this, we also use the fact that $E_{0,n,d}$ does not have fixed part due to the assumption that $E_{0,n,d}$ has invertible Euler class.

We interrupt the flow with a few remarks.

\begin{rmk}
The reason to subtract $t^{j',\chi}(w)$ in equation \eqref{tjx} is to exclude the case when the other end at $Z_{j'}$ is directly connected to another \unbroken of the same fractional character $\chi$ (thus satisfying condition \sstar~ and violating the assignment of decorated graphs).
\end{rmk}

\begin{rmk}
It might appear that there is a missing factor coming from the gluing at the node on the $\sss^-$ side. 
However, it is automatically taken care of because, in the expression $(\text{ev}_-)^*\left[F^{j'}(w)-t^{j',\chi}(w) \right]$, the definition of $F^{j'}(w)-t^{j',\chi}(w)$ already involves the pull-back under $\iota_j$ and the push-forward under certain evaluation maps.
\end{rmk}

\begin{rmk}\label{Kdep1}
Since we pull back along $\ff$, the outcome only depends on the restriction $(TX\oplus E)|_{\XX}$. One can see the corresponding euler classes only depends on the class of $TX\oplus E$ in the $K^0(\XX)$.
\end{rmk}

Now let's reorganize $t^{j,\chi}$. Plugging $w=\psi_--\chi$ into $(ev_-)^*F^{j'}(w)$ in the expression (\ref{tjx}), the term $(ev_-)^*F^{j'}(\psi_--\chi)$ can be expanded using the Taylor expansion
\begin{align*}
(ev_-)^*F^{j'}(\psi_--\chi)= &(ev_-)^*F^{j'}(-\chi)\\
+&\psi_-(ev_-)^*\displaystyle\frac{\partial}{\partial t}\left.F^{j'}(t)\right|_{t=-\chi} + \frac{(\psi_-)^2}{2!}(ev_-)^*\frac{\partial^2}{\partial z^2}\left.F^{j'}(t)\right|_{t=-\chi}+\dotsc .
\end{align*}
As $\psi_-$ is nilpotent in $H^*((\mm))$, the above is a finite sum. 

Applying the above Taylor expansion and expanding $\displaystyle \frac{1}{-z-\psi_++\chi}$ in terms of $(-z+\chi)^{-1}$, we have
\begin{equation}\label{tjx1}
t^{j,\chi}(z)  =  \iota_j^*(ev_+)^{vir}_*\left[   \sum\limits_{i,k,j',d,\chi}\frac{Q^{d}\A  \left.  \displaystyle\frac{\partial^k}{\partial w^k}\left[(ev_-)^*(F^{j'}(w)-t^{j',\chi}(w))  \right]  \right|_{w=-\chi}}{(-z+\chi)^i}   \right] ,
\end{equation}
where $A_{i,j',k,d,\chi}\in H^*(\mm)$ are some classes whose exact expressions does not concern us in this paper.
In a moment, we will show that all poles of $F^j$ at $-\chi$ come from $t^{j,\chi}$, i.e., $\Prin\limits_{w=-\chi}F^{j}(w)=t^{j,\chi}(w)$. 

\begin{caution*}
There might be a confusion of signs. Notice we use $-z$ in $F^j(-z)$ but $w$ in  $F^j(w)$. However, $t^{j,\chi}(z)$ and $t^{j,\chi}(w)$ are used, where both $z$ and $w$ variables carry positive sign. 
As a result, the pole of $F^j (w)$ at $-\chi$ means the pole at $w=-\chi$, while $z=\chi$.
This choice of signs is dictated by the localization expression.
\end{caution*}

Next we consider the case when the vertex $s_1$ is incident with more than one edge. 
Write
\[
 t^j(z)=\iota_j^*t(z) + \sum\limits_\chi t^{j,\chi}(z).
\]
Later $t^j$ will fit into Equation \eqref{Fj}. Let us start by applying virtual localization to Equation \eqref{Fj}:
\begin{align*}
&F^j(-z) \\
=& -1z+\iota_j^*t+\sum\limits_{\vec\Gamma~\text{with}~ p_{s_1}=j} Q^{d_\Gamma}\text{Cont}_{\vec\Gamma}(z)\\
=&-1z+\iota_j^*t +\sum\limits_{\begin{subarray}{c} d,n,\\j_1,\dotsc,j_{n+1} \end{subarray}} \iota_j^*(\text{ev}_1)^{\vir}_* \cdot \\
  & \qquad \cdot  \left[ \displaystyle\frac{Q^d}{n!}\frac{\text{Cont}_{s_1}}{(-z-\psi_1)}  \prod\limits_{i=2}^{n+1}  (\text{ev})_i^*\left\{  \iota_j^*t(\psi_+-\chi)+\sum\limits_{\begin{subarray}{c} \vec{\Gamma} ~\text{with} ~d_{s_1}=0,\\p_{s_1}=j_i, \, n_{s_1}=0, \,  \text{val}(s_1)=1 \end{subarray}} Q^{d_\Gamma}\text{Cont}_{\vec\Gamma}(\psi_+-\chi)  \right\}  \right]
\end{align*}
where $\text{ev}_i$ are the evaluation maps in $\bM_{0,n+1}(Z_j,d)$ and $$\text{Cont}_{s_1}=e_T^{-1}(E_{0,n+1,d}\oplus(N_j)_{0,n+1,d}).$$

This expression follows from localization computation by observing the following points. 
Firstly, the vertex $s_1$ contributes $e_T^{-1}((N_j)_{0,n+1,d})$ to the Euler class of virtual normal sheaf.
Secondly, the contribution coming from smoothing the nodes is included in $\text{Cont}_{\vec\Gamma}(\psi_1-\chi)$.

Since
\[
 t^j=\iota_j^*t  +  \sum\limits_{\begin{subarray}{c} \vec{\Gamma} ~~\text{with} ~d_{s_1}=0, \\p_{s_1}=j, n_{s_1}=0, \text{val}(s_1)=1 \end{subarray}}Q^{d_\Gamma}\text{Cont}_{\vec\Gamma}(z) ,
\]
we have
\begin{align*}
F^j(-z)&=-1z+t^j(z)+\sum\limits_{n,d}\frac{Q^d}{n!}\iota_j^*(ev_1)^{vir}_*\left[\frac{e_T^{-1}(E\oplus (N_j)_{0,n+1,d})}{-z-\psi_1}\prod\limits_{i=2}^{n+1}\text{ev}_i^*t^j(\psi_1) \right] .
\end{align*}

\begin{rmk}\label{Kdep2}
Notice in the above expression, $e_T^{-1}(E\oplus (N_j)_{0,n+1,d})$ only depends on the class $[E\oplus N_j]\in K^0_T(Z_j)$. 
Furthermore, since $\psi_1$ is nilpotent for a fixed degree of Novikov variablse $Q^d$, the last term is a polynomial in $z^{-1}$ since there are finitely many graphs of curve class $\leq d$. 
\end{rmk}

Observe now the only term contributing poles not at $z=0$ or $z =\infty$ is $t^j(z)$. In particular, the pole at $z=-\chi$ is of the following form
\begin{equation}\label{tjx2}
t^{j,\chi}(z)  =  \iota_j^*(ev_+)^{vir}_*\left[   \sum\limits_{i,j',d,k,\chi}\frac{Q^{d}\A  \left.  \displaystyle\frac{\partial^k}{\partial w^k}\left[(ev_-)^* \left( F^{j'}(w) - \Prin\limits_{z=-\chi}F^{j'}(w) \right)  \right]  \right|_{w=-\chi}}{(-z+\chi)^i}   \right]
\end{equation}
This is exactly $\Prin\limits_{z=-\chi}F^j(z)$ in Theorem \ref{main}.

Let us summarize what we have obtained so far.
If $F\in \Ll_E$, there exists $F^j\in\Ll_{N_j\oplus E}^j$, satisfying (1),(2) in Theorem \ref{main}, such that when expanded in $1/z$, $\iota_j^*F=F^j$ as Laurent series in $1/z$. The coefficients $A_{i,j,j',k,\chi}$ only depend on the classes of vector bundle $(TX\oplus E)|_{\XX}$, $N_j\oplus E|_{Z_j}$ in their corresponding K-groups, $\mm$, $[\mm]^{vir}$ and their universal families $\mathcal C_{j,j',d,\chi}$

This is the forward implication of the theorem. 
We now proceed to prove the converse.
Given $\{F^j\}$ satisfying (1), (2) in Theorem \ref{main}, we set 
$$F :=\sum\limits_{j=1}^l (\iota_j)_!\displaystyle\frac{F^j}{e_T(N_j)}.$$ 
It follows from localization that $F^j=\iota_j^*F$.
Our goal is to show that $F\in \Ll_E$. 
First of all, there is always a point $F'\in \Ll_E$ such that $[F']_+=[F]_+$. 
By previous localization computation, we can find $F'^j$ such that $F'^j=\iota_j^*F'$ as formal functions. 
It follows that $[F'^j]_+=[F^j]_+$. But at this moment, it's still unclear whether $[F^j]_\bullet=[F'^j]_\bullet$, $[F^j]_-=[F'^j]_-$ are true.
The subsequent argument shows that both of these are in fact true, and therefore $F=F'\in \Ll_E$. More specifically, we will show that \emph{$[F^j]_\bullet$ and $[F^j]_-$ are uniquely determined by $[F^j]_+$ and the recursion relation} (condition (2) of Theorem \ref{main}), \emph{assuming $F^j\in \Ll^j_{N_j\oplus E}$} (condition (1) of Theorem~\ref{main}). Since $F^j$, $F'^j$ both satisfy conditions (1) and (2) in Theorem~\ref{main}, the uniqueness guarantees that $[F^j]_\bullet=[F'^j]_\bullet$, $[F^j]_-=[F'^j]_-$. To highlight the function space we are working on, the following convention is adapted.

We work on the formal neighborhood of the space $\mathcal{H}^j$ with polarization, defined similarly to that of $\mathcal{H}$ in Convention~\ref{polarization2}. 
The decomposition $\mathcal H^j=\mathcal H^j_+\oplus \mathcal H^j_-$ induces a decomposition $(\mathcal H^j,-1z)=(\mathcal H^j,-1z)_+\oplus (\mathcal H^j,-1z)_-$. 

To simplify the expressions, we slightly abuse the notation in the following way.
\[
  \mathcal H^j := (\mathcal H^j,-1z), \quad \mathcal H^j_+ := (\mathcal H^j,-1z)_+, \quad \mathcal H^j_- := (\mathcal H^j,-1z)_- .
\]

Recall the $S$-matrix in Definition \ref{def:S}. Since $\{F^j\}$ lie in the corresponding Lagrangian cones by assumption, there exists a unique $u^j\in H^j:=H^*(Z_j)$ such that $G^j(z):=S_{u^j}(-z)F^j(z)\in z\mathcal H^j_+$ for a given $j$. (See \cite[\S 2.3]{GB}.)
Notice that $F^j$ can be recovered according to $S_{u^j}^*(z)G^j(z)=F^j(z)$. 
We proceed to show that there is a recursive relation for $G^j$ determined by $[F^j]_+$.  
 
First of all we have the $S$ matrix $S_{u^j}(-z)=1+O(1/z)\in 1+\mathcal H^j_-$, a formal series in $1/z$ but nevertheless a \emph{polynomial} in $1/z$ for each $Q^d$ term. 
It follows from formal manipulation of formal series in $z^{-1}$ that the Taylor polynomial of order $K$ at $z=-\chi$ for the $S$-matrix can be written as
\begin{align*}
&S_{u^j}(-z) = \left( S_{u^j}(\chi)+\displaystyle\frac{\partial}{\partial z}S_{u^j}(\chi)(-z-\chi)+\dotsc+\displaystyle \frac{1}{K!} \frac{\partial^K}{\partial z^K}S_{u^j}(\chi)(-z-\chi)^K \right) + (-z-\chi)^{K+1}R(z^{-1})
\end{align*}
where $R(z^{-1})\in \mathcal H^j_-$ is the remainder term.

Setting
\[
 \bar P_{i,\chi}^j=\iota_j^*(ev_+)^{vir}_*\left[  \sum\limits_{j',k,d} Q^{d}\A  \left.  \displaystyle\frac{\partial^k}{\partial z^k}\left[(ev_-)^*  \left[ S_{u^{j'}}^*(z)G^{j'}(z) - \Prin\limits_{z=-\chi} \left( S_{u^{j'}}^*(z)G^{j'}(z) \right) \right]   \right]  \right|_{z=-\chi} \right] ,
\]
we have 
\[
 \Prin\limits_{z=-\chi}F^j(z)=\sum\limits_{i=0}^{\infty}\frac{\bar P_{i,\chi}^j}{(z+\chi)^i} .
\]
Notice that this sum has finite terms for any fixed power of Novikov variables. 
That is, the coefficient of $Q^d$ has finite order pole . 
Therefore it lives in the function space $\mathcal H^j$.
The following is the recursion relation of $G^j$ promised earlier.

\begin{lem} \label{l:3.11}
We have the following recursion relation of $G^j(z)$
\begin{equation}\label{rec}
G^j(z)=[G^j(z)]_+ + \sum\limits_{\chi}\left( \sum\limits_{i=0}^{\infty}\displaystyle\frac{\bar P_{i,\chi}^j}{(z+\chi)^i}\sum\limits_{l=0}^{i-1}\displaystyle\frac{\partial^i}{\partial z^l}S_{u^j}(\chi)(-z-\chi)^l \right) .
\end{equation}
Notice that this becomes a finite sum for fixed Novikov variables.
\end{lem}

\begin{proof}
The sum $\sum\limits_{l=0}^{i-1}\displaystyle\frac{\partial^i}{\partial z^l}S_{u^j}(\chi)(-z-\chi)^l$ is the Taylor polynomial for $S(-z)$ of order $i-1$. 
Now use the above fact that $S_{u^j}(-z)-\sum\limits_{l=0}^{i-1}\displaystyle\frac{\partial^i}{\partial z^l}S_{u^j}(\chi)(-z-\chi)^l=(-z-\chi)^iR(z^{-1})$. 
Simple calculation shows that the difference of the LHS and of the RHS of Equation \eqref{rec} is a $z^{-1}$ series (and a polynomial in $z^{-1}$ in each coefficient of $Q^d$ with fixed $d$). 
However, since both sides lie in $\mathcal H^j_+$ and so does the difference. Hence the difference must vanish.
\end{proof}

This gives us a recursive relation for $G^j(z)$ with initial condition $[G^j(z)]_+\in H^*(Z_j;S_T)[z][\![Q]\!]$ and $u^j\in H^j$. 
We note that the $[G^j(z)]_+$ term is denoted by $q^\alpha(z)$ in the proof of \cite[Theorem~2]{GB}.

The rest of the proof is identical to the corresponding part of \emph{loc.\ cit.}.
Setting $Q=0$, $\iota_j^*t$ are related to $[G^j(z)]_+$ in the following way
\[
-z+\iota_j^*t=\left[S^{-1}_{u^j}(-z)[G^j(z)]_+ \right]_+ .
\]
We are left to show $\{u^j\}$ are determined by $\{\iota_j^*t\}$ in view of the fact that $G^j(z)\in z\mathcal H^j_+$, cf.\ \cite[\S 2.3]{GB}. Again modulo Novikov variables, we have $G^j(z)=[G^j(z)]_+$, $S_{u^j}(-z)=e^{-u^j/z}$.
\[
[G^j(z)]_+=\left[e^{-u^j/z}(z+\iota_j^*t(-z))\right]_+\in z\mathcal H^j_+ .
\]
The constant term being zero provides us with an equation between $u^j$ and $\iota_j^*t(z)$. 
Using formal implicit function theorem, one can see that  
$u^j$ is uniquely determined by successive $Q$-adic approximations.

\section{Gromov-Witten invariants of projective bundles}\label{projectivebundle}

Let $Y$ be a nonsingular variety equipped with an algebraic action by $T=(\C^*)^m$. 

\begin{defn} \label{GKM}
$Y$ is said to be an {\it algebraic GKM manifold} if it has finitely many fixed points and finitely many one-dimensional orbits under the $T$ action.
\end{defn}

In this section, we assume $Y$ to be proper algebraic GKM manifold. In this case the closure of any one-dimensional orbit in $Y$ is $\Pp^1$.
Examples include proper toric varieties, partial flag varieties including Grassmannians, etc.

Let $V$ be a $T$-equivariant vector bundle of rank $r$ over $Y$ and let $c_T(V) (x) := \sum c_i(V) x^{r-i}$ be the total Chern polynomial.
The classical result gives a presentation of the equivariant cohomology of the projective bundle $\Pp_Y(V)$ as 
$$H_T^*(\Pp(V)) = \frac{H^*_T(Y)[h]}{(c_T(V)(h))}$$ 
where $h := c_1(\sO(1))$.
Let $V_1, V_2$ be two equivariant vector bundles over $Y$ having the same equivariant Chern classes, i.e., $c_T(V_1)=c_T(V_2)$.
Their equivariant cohomology rings are canonically isomorphic using the above presentation
\[
\FF:H_T^*(\Pp(V_1))\cong H^*_T(\Pp(V_2)),
\]
where the isomorphism $\FF$ sends corresponding $c_1(\sO(1))$ to $c_1(\sO(1))$. 
Furthermore, $\FF$ induces an isomorphism between $N_1(\Pp(V_1))$ and $N_1(\Pp(V_2))$ by Poincar\'e pairing. By a slight abuse of notations, we still denote the induced isomorphism by $\FF$. In other words, given a curve class $\beta\in N_1(\Pp(V_1))$, $\FF\beta$ is the unique curve class in $N_1(\Pp(V_2))$ such that $(\FF D,\FF\beta)=(D,\beta)$ for any $D\in N^1(\Pp(V_1))$.

The following is the main result of this paper.

\begin{thm}\label{main2}
The isomorphism $\FF$ between equivariant cohomology rings and numerical curve classes induces an isomorphism of the full genus zero equivariant Gromov-Witten invariants between $\Pp(V_1)$ and $\Pp(V_2)$. More precisely,
\[
 \langle\psi^{a_1}\sigma_1,\dotsc,\psi^{a_n}\sigma_n\rangle_{0,n,\beta}=\langle\psi^{a_1}\FF\sigma_1,\dotsc,\psi^{a_n}\FF\sigma_n\rangle_{0,n,\FF\beta}.
\]
\end{thm}

The rest of this section is devoted to the proof of this theorem. 
By definition of algebraic GKM manifolds, the fixed loci of $\Pp(V_i)$ are the fibers over the $T$-fixed points of $Y$. 
Let $p$ be a fixed point in $Y$. For a given $1\leq j\leq l$, Since $c_T(V_1)=c_T(V_2)$, we have
\[
 c_T(V_1|_{p})=c_T(V_2|_{p})\in R_T=H^*_T(\pt),
\]
where $V_i|_p$ are the fibers at $p$. 
Interpreting $R_T$ as the representation ring of $T$, the equality is also saying that $V_i|_{p}$ are isomorphic as $T$-representations. 
Thus $\Pp(V_1)|_{p}, \Pp(V_2)|_{p}$ are $T$-equivariantly isomorphic. 
Hence, the fixed loci of $\Pp(V_1)$ and $\Pp(V_2)$ are naturally identified. 

Given $Z\subset\Pp(V_1)|_{p}, Z'\subset\Pp(V_2)|_{p}$ two fixed loci that are identified under $\FF$. 
One sees that there is a $T$-equivariant short exact sequence 
\[
 0\rightarrow N_{Z/(\Pp(V_1)|_p)} \rightarrow N_{Z/\Pp(V_1)} \rightarrow \pi^*N_{\{p\}/Y} \rightarrow 0 .
\]
Thus $N_{Z/\Pp(V_1)}$ can be $T$-equivariantly deformed to $N_{Z/(\Pp(V_1)|_p)}\oplus \pi^*N_{\{p\}/Y}$, where $\pi: \Pp(V_1)\rightarrow Y$ is the projection, by sending the $T$-equivariant extension class to zero.
Similar observations conclude that $N_{Z/\Pp(V_1)}$ and $N_{Z'/\Pp(V_2)}$ represent the same element in $K^0_T(Z)\cong K^0_T(Z')$.

Now in order to identify Lagrangian cones $\Ll_{\Pp(V_1)}$ and $\Ll_{\Pp(V_2)}$, it suffices to show that the recursion relations (2) in Theorem \ref{main} are identified under $\FF$.
Let's analyze the one-dimensional orbit closures in $Y$ first.

\begin{lem}\label{orbitlem}
Let $p\in Y$ be a fixed point and $l_1, l_2$ be two distinct 1-dimensional orbit closures passing through $p$.
The (integral) characters $\chi_1,\chi_2$ of $l_1, l_2$ at $p$ do not lie on the same ray in $C(T)_\Q$ .
\end{lem}

\begin{proof}
Suppose we have these two orbit closures $l_1, l_2$ such that $\chi_1, \chi_2$ lie on the same ray in $C(T)_\Q$. 
Write $\chi_0$ to be the first $\Z$-point along the ray of $\chi_1$ and $\chi_2$ in $C(T)_\Q$. 
Decompose $T_pX$ into irreducible $T$-representations $\bigoplus\limits_{\chi\in C(T)}T_\chi$. The assumption implies that $\bigoplus\limits_{a\in\Z_{\geq 0}}T_{a\chi_0}$ is at least of dimension $2$. 

Choosing $\rho$ such that $(\rho,\chi_0)>0$, there is a subvariety of $Z^{\rho}$ containing $p$ that is $T$-isomorphic to $\bigoplus\limits_{(\rho,\chi)>0} T_{\chi}$ by \cite[Theorem 2.5]{bb} (note that $p$ is a discrete fixed point). $Z^{\rho}$ contains the subrepresentation $\bigoplus\limits_{a\in\Z_{\geq 0}}T_{a\chi_0}$.
However, an elementary verification shows that there are infinitely many $1$-dimensional orbits in $\bigoplus\limits_{a\in\Z_{\geq 0}}T_{a\chi_0}$ since its dimension is at least $2$. 
This is a contradiction to the assumption that $Y$ is algebraic GKM.
\end{proof}

From now on, we are mostly concerned with general properties of a projective bundle over an algebraic GKM manifold.
To simplify the notation, set $\pi: X:=\Pp_Y(V) \to Y$.
Recall the notations defined in Section~\ref{s:2}. 
$Z_1,\dotsc,Z_n$ are the fixed loci of $X$ under $T$-action, and $\XX=\ff(\cc)$ can be seen as the union of the images of all \unbrokens for fixed $j,j',d,\chi$.
(See Notations~\ref{n:2.14}.)

\begin{lem}
Fix $\chi\in C(T)_\Q$ and $Z_j\subset X$ such that $\pi(Z_j)=\{p_j\}\subset Y$. For any $d\in \NE(X)$ and $Z_{j'}$, we have either $\mm=\emptyset$ or $\pi\circ\ff(\cc)\subset l$, where $l\subset Y$ is the closure of a one-dimensional orbit.
\end{lem}
\begin{proof}
Fix $\chi$ and $Z_j$ and pick any $d, Z_{j'}$. If $\mm\neq\emptyset$, $\pi(\XX)$ must be either the fixed point $\{p_j \}$ or some one dimensional orbit closure $l\subset Y$. In the latter case, by definition $\chi$ must be proportional to the fractional character of $l$ at $p_j$. By the previous lemma, there is at most one orbit closure $l$ passing through $p_j$ such that its fractional character is proportional to $\chi$. Since $\chi$ is fixed at the beginning, such a universal $l$ can be easily found for all $d$ and $Z_{j'}$.
\end{proof}

Let $l$ be a one dimensional orbit closure in $Y$ and $X|_l := \pi^{-1}(l) \subset X$. 
With the help of the previous lemma, when there is a fixed $\chi$ and $Z_j$, we just found an $l\subset Y$ such that in the recursion condition for $\Ll_X$, all possible \unbroken that are used in computing $\Prin\limits_{z=-\chi}F^j(z)$ are contained in $X|_l$. We are left to match the recursion conditions on $\Prin\limits_{z=-\chi}F^j(z)$ when we change the bundle from $V_1$ to $V_2$. 
For that purpose, we need to make sense of $\Ll_{X|_l,N_{X|_l/X}}$. 
We follow Remark \ref{invertible} by adding an auxiliary $\C^*$ action that acts on $X|_l$ trivially, but on $N_{X|_l/X}$ by scaling. 
Let $R_{T\times\C^*}=R_T[\x]$ where $\x$ is the equivariant parameter for the extra $\C^*$ action. Let $R_T[\x,\x^{-1}]\!]$ be the ring of Laurent series in $\x^{-1}$. 
Let $\pi_{0,n,d}:\mathcal C_{0,n,d}\rightarrow \bM_{0,n}(X|_l,d)$ be the universal curve and $f_{0,n,d}:\mathcal C_{0,n,d}\rightarrow X$ be the universal stable map. 
Let
\[
  0\rightarrow N^0\rightarrow N^1\rightarrow 0
\]
be a two term complex of locally free sheaves, whose cohomology are $R^i (\pi_{0,n,d})_*f_{0,n,d}^*N_{X|_l/X}$. 
Recall we define
\[
 e_{T\times\C^*}((N_{X|_l/X})_{0,n,d})=\displaystyle\frac{e_{T\times\C^*}(N^0)}{e_{T\times\C^*}(N^1)}\in (H^*_{T}(\bM_{0,n}(X|_l,d))[\x])_\text{loc} 
\]
where $(H^*_{T}(\bM_{0,n}(X|_l,d))[\x])_\text{loc}$ is the localization of $H^*_{T}(\bM_{0,n}(X|_l,d))[\x]$ by inverting all monic polynomials in $\x$. 
Embedding it into $H^*_{T}(\bM_{0,n}(X|_l,d)\otimes_{R_T}R_T[\x, \x^{-1}]\!]$ as a subspace renders $\Ll_{X|_l,N_{X|_l/X}}$ well defined under the extra $\C^*$ action. 
We will prove, in Proposition~\ref{inv}, that the $\x=0$ limit exists in the sense of Remark~\ref{invertible}.

Since $N_{X|_l/X}=\pi^*N_{l/Y}$, we first analyze the $T$ action on $N_{l/Y}$.

\begin{lem}
There is a splitting $T=T'\times \C^*$ such that $T'$ acts on $l$ trivially. Furthermore, there is no nontrivial $T'$-fixed subsheaf in $N_{l/Y}$.
\end{lem}

\begin{proof}
As we have seen in Corollary~\ref{c:2.5}, the one dimension orbit is isomorphic to $\C^*$. 
Picking any point on it, we have an induced group homomorphism $T\rightarrow \C^*$. Let $K$ be its kernel. $Hom(-,\C^*)$ is an exact functor among abelian groups because $\C^*$ is injective. $Hom(K,\C^*)$ decomposes into the product of a free abelian group and a torsion one. Therefore $K$ decomposes into $K'\times K''$ where $K'\cong (\C^*)^{m-1}$ and $K''$ is a finitely generated torsion abelian group. $K'$ is naturally embedded in $K$. Using $K'$ instead of $K$, we have a short exact sequence
\[
 0\rightarrow K'\rightarrow T\rightarrow \C^*\rightarrow 0 .
\]
Since $K'$ fixes points in $l$, there is an induced action of the $\C^*$ on $l$. 
Now with $K'$ being a sub-torus, $T$ splits into $K'\times \C^*$ due to injectivity of $K'$. 
We have shown the first statement with $T'=K'$.

We now proceed to the second statement. 
First observe that $N_{l/Y}$ can be decomposed into $T'$-eigensheaves: 
\begin{equation}\label{splitN}
N_{l/Y}=\bigoplus_{\chi\in C(T')}N_{\chi}
\end{equation}
such that ach component is locally free. 
It therefore suffices to think about $N_{l/Y}|_{\{p\}}$ where $p$ is a fixed point on $l$. 
Similar to Lemma \ref{orbitlem}, decompose $T_pY$ into irreducible $T$-representations $\bigoplus\limits_{\chi\in C(T)}T_\chi$. The tangent direction along $l$ corresponds to one of the factor $T_{\chi_0}$ for some $\chi_0\in C(T)$. This $\chi_0$ is proportional to the character corresponding to $T\rightarrow \C$ from the above short exact sequence. Notice that $N_{l/Y}|_{\{p\}}$ is $T$-isomorphic to $T_pY/T_{\chi_0}$. Since $Y$ has isolated fixed point and $1$-dimensional orbits, $T_0=0, T_{\chi_0}$ is $1$-dimensional. Hence there is no fixed subspace in $T_pY/T_{\chi_0}$.
\end{proof}

We are left to analyze $e_{T\times\C^*}((N_{X|_l/X})_{0,n,d})^{-1}$. Let $R_{T\times\C^*}=R_T[\x]$ and $R_T[\x,\x^{-1}]\!]$ be the ring of Laurent series in $\x^{-1}$. 

By definition, $e_{T\times\C^*}((N_{X|_l/X})_{0,n,d})^{-1} \in (H^*_{T}(\bM_{0,n}(X|_l,d))[\x])_\text{loc}$.

\begin{prop}\label{inv}
\[
 e_{T\times\C^*}((N_{X|_l/X})_{0,n,d})^{-1} \in (H^*_{T}(\bM_{0,n}(X|_l,d))[\x])_\text{loc}
\]
has $\x=0$ limit in $H^*_{T}(\bM_{0,n}(X|_l,d))\otimes_{R_T} S_T$ in the sense of Remark \ref{invertible}.
\end{prop}

Since we are working on the invariant subvariety $X|_l$, in the rest of this proof, we replace the whole variety $X$ by $X|_l$, and still use $\bM_{\vec\Gamma}$, $\mm$, etc for the corresponding notions in $X|_l$.

Let $\pi_{\vec\Gamma}:\mathcal C_{\vec\Gamma}\rightarrow \bM_{\vec\Gamma}$ be the universal family over $\bM_{\vec\Gamma}$ and $f_{\vec\Gamma}:\mathcal C_{\vec\Gamma}\rightarrow X_l$ be the universal stable map.

\begin{lem}
$R^i (\pi_{\vec\Gamma})_*f^*_{\vec\Gamma}(N_{X|_l/X})$ are locally free for $i=0,1$.
\end{lem}

\begin{proof}
This amounts to saying that, given an invariant stable map $f:C\rightarrow X_l$ with $n$ markings assigned with a decorated graph $\vec\Gamma$, the dimensions of $H^i(C,f^*N_{X|_l/X})$ for $i=0,1$ only depend on the graph $\vec\Gamma$ (i.e., when the stable map varies in $\bM_{\vec\Gamma}$, these dimensions are constant). 

Let $v$ be a vertex in $\vec\Gamma$. Suppose $v$ corresponds to a sub-curve $C_v$. Under the projection $\pi:X|_l\rightarrow l$, $C_v$ has to map to a point. Therefore $N_{X|_l/X}|_{C_v}$ is a trivial sheaf and the dimension of $H^0(C,f^*N_{X|_l/X})$ is the rank of $N_{X|_l/X}$. There is no $H^1$ since we are working on genus-$0$ curves. If $v$ corresponds to a point in $C$, the dimension of $H^0$ is also the rank of $N_{X|_l/X}$.

Let $e$ be an edge in $\vec\Gamma$ and $C_e$ be the sub-curve corresponding to $e$. There are two possibilities of the image of $f(C_e)$ under the projection $\pi:X|_l\rightarrow l$. Recall $d_e\in \NE(X|_l)$ is the curve class corresponding to $f(C_e)$. 

\begin{enumerate}
\item If $\pi_*d_e=0$, $\pi\circ f(C_e)$ is a point. And the same thing as the vertex case happens. 
\item If $\pi_*d_e\neq 0$, $\pi\circ f(C_e)$ is $l$. Recall $C_e$ is a chain of $\Pp^1$. Let $C_1,\dotsc,C_k$ be the irreducible components of $C_e$. In this case, there is only one component $C_i$ such that $\pi\circ f(C_i)=l$ (otherwise the fractional characters on the components violate $C_e$ being an \unbroken). Suppose $\pi_*d_e=k[l]$. Then $\pi\circ f:C\rightarrow l$ is a degree $k$ cover between $\Pp^1$'s. Thus $H^i(C_i,f^*N_{X|_l/X})=H^i(C_i,f^*\pi^*N_{l/Y})$ depends only on $k$ and the bundle $N_{l/Y}$ over $l$. Other components than $C_i$ are mapped to points under $\pi\circ f$ and won't contribute extra dimensions to $H^i(C,f^*N_{X|_l/X})$. 

\end{enumerate}

The dimensions of $H^i(C,f^*N_{X|_l/X})$ can then be calculated by passing to the normalization of $C$. To sum it up, $\text{dim}(H^i(C,f^*N_{X|_l/X}))$ depends only on the decorated graph $\vec\Gamma$ assigned to the stable map $f:C\rightarrow X_l$ and the normal bundle $N_{l/Y}$ on $l$. This implies the lemma we are proving.
\end{proof}

As a result, we have

$$\left. e_{T\times\C^*}((N_{X|_l/X})_{0,n,d})^{-1} \right|_{\bM_{\vec\Gamma}}=\displaystyle\frac{e_{T\times\C^*}(R^1 (\pi_{\vec\Gamma})_*f^*_{\vec\Gamma}(N_{X|_l/X}))}{e_{T\times\C^*}(R^0 (\pi_{\vec\Gamma})_*f^*_{\vec\Gamma}(N_{X|_l/X}))} $$

We are left to prove this quotient has $\x=0$ limit. 

\begin{lem} \label{l:4.8}
$e_{T}(R^i (\pi_{\vec\Gamma})_*f^*_{\vec\Gamma}(N_{X|_l/X}))$ are invertible elements in $H^*(\bM_{\vec\Gamma})\otimes_{\C}S_T$ (notice that we are without the $\C^*$ action).
\end{lem}

\begin{proof}
Notice we have 
\[
 H^*_T(\bM_{\vec\Gamma})\otimes_{R_T}\otimes R_T=R_T\otimes_\C H^*(\bM_{\vec\Gamma}) .
\]
One can decompose 
\[
 e_T(R^i (\pi_{\vec\Gamma})_*f^*_{\vec\Gamma}(N_{X|_l/X}))  =  \sigma_0\otimes 1+\sum\limits_{j\geq 1} \sigma_j\otimes \tau_j \in R_T\otimes_\C H^*(\bM_{\vec\Gamma})
\]
where $\sigma_j\in R_{T}$ and $\tau_j\in H^{>0}(\bM_{\vec\Gamma})$. If we can prove $\sigma_0\neq 0$, the expression above is invertible after tensoring $S_T$. In order to do this, we show that $R^i (\pi_{\vec\Gamma})_*f^*_{\vec\Gamma}(N_{X|_l/X})$ do not have non-trivial fixed subsheaves. It suffices to prove it fiberwise, and can be reduced to the following.

\begin{claim}
For any proper irreducible invariant curve $C\subset X|_l$, $H^i(C, N_{X|_l/X})$ does not have any nontrivial fixed subspace for $i=0,1$.
\end{claim}

\begin{proof}
It is because of the splitting in equation (\ref{splitN}) and the fact that  $N_{X|_l/X}=\pi^*N_{l/Y}$, and there are two cases. 
\begin{enumerate}
\item If $\pi(C)=l$, since $\pi$ is equivariant, $T'$ acts on $C$ trivially. $\left. N_{X|_l/X}\right|_C$ splits into eigensheaves where $T'$ acts by scaling. Since $N_0=0$, the claim follows. 
\item If $\pi(C)$ is a fixed point, although $T'$ might act on $C$ non-trivially, the underlying vector bundle $\left. N_{X|_l/X}\right|_C$ is trivial since it is a pull-back from a point. Only $H^0(C, N_{X|_l/X})$ can be nonzero, and it consists of constant sections. $T'$ acts on it without trivial sub-representation again due to $N_0=0$.
\end{enumerate}
\end{proof}

As a result, $\sigma_0\neq 0$. $e_T(R^i (\pi_{\vec\Gamma})_*f^*_{\vec\Gamma}(N_{X|_l/X}))$ are invertible because $H^*(\bM_{\vec\Gamma})$ is finite dimensional.
\end{proof}

\begin{proof} (of Proposition~\ref{inv})
Now that Lemma~\ref{l:4.8} is proven, the $\x=0$ limit for $e_{T\times\C^*}((N_{X|_l/X})_{0,n,d})$ can be readily constructed via localization formula. 
The limit is invertible because each piece $e_{T}(R^i (\pi_{\vec\Gamma})_*f^*_{\vec\Gamma}(N_{X|_l/X}))$ on each connected component of $\bM_{\vec\Gamma}$ is invertible. 
\end{proof}

Finally, we use the following fact.

\begin{lem}\label{P1}
Any two equivariant vector bundles over $\Pp^1$ with the same equivariant Chern classes can be connected by an equivariant family of vector bundles over a connected base.
\end{lem}

\begin{proof}
This lemma amounts to saying that the moduli stack of equivariant vector bundles with fixed Chern class over $\Pp^1$ is connected. This can be proven by, for example, adding a framing condition to form the fine moduli of framed equivariant vector bundles (\cite{P}). This forms a fine moduli space which is a covering of the moduli stack of equivariant vector bundles with connected fibers. By the presentation for moduli of framed toric bundles given in \cite{P}, in the case of $\Pp^1$, the moduli space is just a product of flag varieties which are obviously connected.
\end{proof}

To sum it up, Lemma~\ref{P1} shows that $\Pp(V_1)|_l$ can be equivariantly deformed to $\Pp(V_2)|_l$. Thus they have the same equivariant Gromov-Witten theory. Moreover, $N_{\Pp(V_i)|_l/\Pp(V_i)}$ are obviously the pull-back of the same vector bundle $N_{l/Y}$ over $l$. As a result, the Lagrangian cones $\Ll_{\Pp(V_i)|_l, N_{\Pp(V_i)|_l/\Pp(V_i)}}$ are identified for $i=1,2$. Using Corollary \ref{A} with $X=\Pp(V_i), W=\Pp(V_i)|_l$, we are finally able to identify the recursion relations (2) for any fixed $Z_j$ and $\chi$ in Theorem \ref{main} between points in $\Ll_{\Pp(V_1)}$ and $\Ll_{\Pp(V_2)}$.
Theorem~\ref{main2} is now proven.

\end{document}